\documentclass[11pt,a4paper]{amsart}
\usepackage[foot]{amsaddr}

\usepackage[bookmarksopen]{hyperref}

\usepackage{xcolor}
\hypersetup{
    colorlinks,
    linkcolor={red!50!black},
    citecolor={blue!50!black},
    urlcolor={blue!80!black}
}

\usepackage[utf8]{inputenc}
\usepackage[english]{babel}
\usepackage{amsmath,amsthm,amsfonts,amssymb}
\usepackage{lmodern}

\usepackage{mathtools}

\theoremstyle{plain}
\newtheorem{theorem}{Theorem}
\newtheorem*{theorem*}{Theorem}

\newtheorem{cor}[theorem]{Corollary}
\newtheorem*{conjecture}{Conjecture}
\newtheorem{lemma}[theorem]{Lemma}
\newtheorem*{lemma*}{Lemma}
\theoremstyle{definition}
\newtheorem{definition}[theorem]{Definition}
\theoremstyle{remark}
\newtheorem{example}{Example}
\newtheorem{rmk}{Remark}

\usepackage{bm}	
\DeclareMathAlphabet{\mathbbold}{U}{bbold}{m}{n}	

\usepackage[left=3cm,right=3cm,top=3cm,bottom=3cm]{geometry}

\DeclareMathOperator{\sinc}{\mathrm{sinc}}
\DeclareMathOperator{\spn}{\mathrm{span}}
\DeclareMathOperator{\rank}{\mathrm{rank}}

\newcommand{\ball}{\mathcal{B}}
\newcommand{\R}{\mathbb{R}}
\newcommand{\N}{\mathbb{N}}
\newcommand{\distr}{\mathcal{D}}

\title{Measure contraction properties of Carnot groups}

\author{Rizzi Luca}
\address{CNRS, CMAP \'Ecole Polytechnique, and \'Equipe INRIA GECO Saclay \^Ile-de-France, Paris}
\email{\href{mailto:luca.rizzi@cmap.polytechnique.fr}{luca.rizzi@cmap.polytechnique.fr}}


\date{\today}
\begin{document}

\begin{abstract}
We prove that any corank 1 Carnot group of dimension $k+1$ equipped with a left-invariant measure satisfies the $\mathrm{MCP}(K,N)$ if and only if $K \leq 0$ and $N \geq k+3$. This generalizes the well known result by Juillet for the Heisenberg group $\mathbb{H}_{k+1}$ to a larger class of structures, which admit non-trivial abnormal minimizing curves.

The number $k+3$ coincides with the geodesic dimension of the Carnot group, which we define here for a general metric space. We discuss some of its properties, and its relation with the curvature exponent (the least $N$ such that the $\mathrm{MCP}(0,N)$ is satisfied). We prove that, on a metric measure space, the curvature exponent is always larger than the geodesic dimension which, in turn, is larger than the Hausdorff one. When applied to Carnot groups, our results improve a previous lower bound due to Rifford. As a byproduct, we prove that a Carnot group is ideal if and only if it is fat.


\end{abstract}

\maketitle

\section{Summary of the results}

Let $(X,d)$ be a length space, that is a metric space such that $d(x,y) = \inf_\gamma \ell(\gamma)$ for all $x,y \in X$, where $\ell(\gamma)$ denotes the length of $\gamma$ and the infimum is taken over all rectifiable curves from $x$ to $y$. Throughout this article we assume that $(X,d)$ has \emph{negligible cut loci}, i.e. for any $x \in X$ there exists a negligible set $\mathcal{C}(x)$ and a measurable map $\Phi^x : X \setminus \mathcal{C}(x) \times [0,1] \to X$, such that the curve $\gamma(t)=\Phi^x(y,t)$ is the unique minimizing geodesic from $x$ with $y$. Moreover, let $\mu$ be a Borel measure such that $0< \mu(\ball(x,r)) < +\infty$ for any $r >0$, where $\ball(x,r)$ is the metric ball of radius $r$ centered in $x$. A triple $(X,d,\mu)$ satisfying the assumptions above is called a \emph{metric measure space}. Any complete Riemannian manifold, equipped with its Riemannian measure, provides an example.

For any set $\Omega$, we consider its geodesic homothety of center $x \in X$ and ratio $t \in [0,1]$:
\begin{equation}\label{eq:omot}
\Omega_t:= \{\Phi^x(y,t) \mid y \in X \setminus \mathcal{C}(x)\}.
\end{equation}
For any $K \in \R$, define the function
\begin{equation}\label{eq:sturm}
s_K(t):= \begin{cases}
(1/\sqrt{K}) \sin(\sqrt{K} t) & \text{if } K>0, \\
t & \text{if } K=0, \\
(1/\sqrt{-K}) \sinh(\sqrt{-K} t) & \text{if } K<0.
\end{cases}
\end{equation}
\begin{definition}[Ohta\footnote{Ohta defines the measure contraction property for general length spaces, possibly with non-negligible cut loci. Under our assumptions, this simpler definition is equivalent to Ohta's, see \cite[Lemma 2.3]{Ohta-MCP}.} \cite{Ohta-MCP}]\label{d:MCP}
Let $K \in \R$ and $N>1$, or $K \leq 0$ and $N=1$. We say that $(X,d,\mu)$ satisfies the \emph{measure contraction property} $\mathrm{MCP}(K,N)$ if for any $x \in M$ and any measurable set $\Omega$ with  with $0< \mu(\Omega)< + \infty$ (and with $\Omega \subset \ball(x,\pi\sqrt{N-1/K})$ if $K > 0$)
\begin{equation}\label{eq:MCP}
\mu(\Omega_{t}) \geq \int_{\Omega} t \left[\frac{s_K(t d(x,z)/\sqrt{N-1})}{s_K(d(x,z)/\sqrt{N-1})}\right]^{N-1} d\mu(z) , \qquad \forall t \in [0,1],
\end{equation}
where we set $0/0 = 1$ and the term in square bracket is $1$ if $K \leq 0$ and $N=1$.
\end{definition}
In this setting, the measure contraction property is a global control on the evolution of the measure of $\Omega_t$. The function $s_{K}$ comes from the exact behavior of the Jacobian determinant of the exponential map on Riemannian space forms of constant curvature $K$ and dimension $N$, where \eqref{eq:MCP} is an equality. On a complete $n$-dimensional Riemannian manifold $M$ equipped with the Riemannian measure, the $\mathrm{MCP}(K,n)$ is equivalent to $\mathrm{Ric}  \geq K$. (see \cite{Ohta-MCP}). Thus, the measure contraction property is a synthetic replacement for Ricci curvature bounds on more general metric measure spaces, and is actually one the weakest. It has been introduced independently by Ohta \cite{Ohta-MCP} and Sturm \cite{S-II}. See also \cite{S-I,S-II,LV-Ricci} for other (stronger) synthetic curvature conditions, including the popular geometric curvature dimension condition $\mathrm{CD}(K,N)$. An important property, shared by all these synthetic conditions, is their stability under (pointed) Gromov-Hausdorff limits.

\medskip

It is interesting to investigate whether the synthetic theory of curvature bounds can be applied to sub-Riemannian manifolds. These are an interesting class of metric spaces, that generalize Riemannian geometry with non-holonomic constraints. Even though sub-Riemannian structures can be seen as Gromov-Hausdorff limits of sequences of Riemannian ones with the same dimension, these sequences have Ricci curvature unbounded from below (see example in \cite{Rifford}). In general, this is due to the fact that the limit $(X,d)$ of a convergent Gromov-Hausdorff sequence of complete, $n$-dimensional Riemannian manifolds with curvature bounded below has Hausdorff dimension $\dim_H(X) \leq n$ (see \cite[Section 3.10]{MR2307192}), but the Hausdorff dimension of sub-Riemannian structures is always strictly larger than their topological one. For this reason a direct analysis is demanded. 

In this paper we focus on Carnot groups. In the following, any Carnot group $G$ is considered as a metric measure space $(G,d,\mu)$ equipped with the Carnot-Carath\'eodory distance $d$ and a left-invariant measure $\mu$. The latter coincides with the Popp \cite{montgomerybook,BR-Popp} and with the Hausdorff one \cite{ABB-Hausdorff}, up to a constant rescaling. All of them coincide with the Lebesgue measure when we identify $G \simeq \R^{n}$ in a set of exponential coordinates. 

\subsection{The Heisenberg group} In \cite{Juillet}, Juillet proved that the $2d+1$ dimensional Heisenberg group $\mathbb{H}_{2d+1}$ does not satisfy the $\mathrm{CD}(K,N)$ condition, for any value of $K$ and $N$. On the other hand, it satisfies the $\mathrm{MCP}(K,N)$ if and only if $K \leq 0$ and $N \geq 2d+3$. 

The number $\mathcal{N} = 2d+3$, which is the lowest possible dimension for the synthetic condition $\mathrm{MCP}(0,N)$ in $\mathbb{H}_{2d+1}$, is surprisingly larger than its topological dimension ($2d+1$) or the Hausdorff one ($2d+2$). This is essentially due to the fact that, letting $\Omega = \ball(x,1)$, we have $\Omega_t \subset \ball(x,t)$ strictly, and 
\begin{equation}
\mu(\Omega_t) \sim \kappa_1 t^{2d+3}, \qquad\text{while} \qquad \mu(\ball(x,t))\sim \kappa_2  t^{2d+2},
\end{equation}
for $t \to 0^+$ and some constants $\kappa_1$ and $\kappa_2$, see \cite[Remark 2.7]{Juillet}.

\subsection{Corank 1 Carnot groups} 

Our first result is an extension of the $\mathrm{MCP}$ results of \cite{Juillet} to any corank $1$ Carnot group. Observe that these structures have negligible cut loci.

\begin{theorem}\label{t:corank1}
Let $(G,d,\mu)$ be a corank $1$ Carnot group of rank $k$. Then it satisfies the $\mathrm{MCP}(K,N)$ if and only if $K \leq 0$ and $N \geq k+3$.
\end{theorem}
\begin{rmk}\label{r:ideal}
We stress that, in general, corank $1$ Carnot groups admit non-trivial abnormal minimizing curves (albeit not strictly abnormal ones). In particular they are not all ideal.
\end{rmk}

\subsection{The geodesic dimension}

The geodesic dimension was introduced in \cite{ABR-Curvature} for sub-Riemannian structures. We define it here in the more general setting of metric measure spaces (which, we recall, are assumed having negligible cut loci).
\begin{definition}
Let $(X,d,\mu)$ be a metric measure space. For any $x \in X$ and $s > 0$, define
\begin{equation}\label{eq:criticalratio}
C_s(x):=\sup\left\lbrace \limsup_{t \to 0^+} \frac{1}{t^s}\frac{\mu(\Omega_t)}{\mu(\Omega)} \mid \Omega \text{ measurable, bounded, $0<\mu(\Omega)<+\infty$}\right\rbrace,
\end{equation}
where $\Omega_t$ is the homothety of $\Omega$ with center $x$ and ratio $t$ as in \eqref{eq:omot}. We define the \emph{geodesic dimension} of $(X,d,\mu)$ at $x \in X$ as the non-negative real number
\begin{equation}\label{eq:gddefs}
\mathcal{N}(x)  := \inf \{s > 0 \mid C_s(x) = +\infty \} = \sup \{s > 0 \mid C_s(x) = 0 \},
\end{equation}
with the conventions $\inf \emptyset = + \infty$ and $\sup \emptyset = 0$.
\end{definition}
Roughly speaking, the measure of $\mu(\Omega_t)$ vanishes at least as $t^{\mathcal{N}(x)}$ or more rapidly, for $t \to 0$. The two definitions in \eqref{eq:gddefs} are equivalent since $s \geq s'$ implies $C_s(x) \geq C_{s'}(x)$. 
\begin{rmk}\label{r:changeofmeas}
$\mathcal{N}(x)$ does not change if we replace $\mu$ with any commensurable measure (two measures $\mu,\nu$ are commensurable if they are mutually absolutely continuous, i.e. $\mu \ll \nu$ and $\nu \ll \mu$, and the Radon-Nikodym derivatives $\tfrac{d\mu}{d\nu}$, $\tfrac{d\nu}{d\mu}$ are locally essentially bounded).
\end{rmk}
The geodesic dimension $\mathcal{N}(x)$ is a local property. In fact, for sufficiently small $t>0$, the set $\Omega_t$ lies in an arbitrarily small neighborhood of $x$. The next theorem puts it in relation with the Hausdorff dimension $\dim_H(B)$ of a subset $B \subseteq X$ (see \cite{AT-Metric} for reference).
\begin{theorem}\label{t:basicprop}
Let $(X,d,\mu)$ be a metric measure space.  Then, for any Borel subset $B$
\begin{equation}
\sup \{\mathcal{N}(x) \mid x \in B\} \geq \dim_H(B).
\end{equation}
\end{theorem}
The next result appears in \cite[Proposition 5.49]{ABR-Curvature}, and we give a self-contained proof. A measure on a smooth manifold is smooth if it is defined by a positive smooth density.
\begin{theorem}\label{t:basicprop2}
Let $(X,d,\mu)$ be a metric measure space defined by an equiregular sub-Rieman\-nian or Riemannian structure, equipped with a smooth measure $\mu$. Then
\begin{equation}
\mathcal{N}(x) \geq \dim_H(X) \geq \dim(X), \qquad \forall x \in X,
\end{equation}
and both equalities hold if and only if $(X,d,\mu)$ is Riemannian.
\end{theorem}
\begin{rmk}\label{r:haudorffcomm}
For an equiregular (sub-)Riemannian structure, the Hausdorff measure is commensurable with respect to any smooth one \cite{Mitchell}. This is no longer true in the non-equiregular case \cite{GJ-Hausdorff}. By choosing the Hausdorff measure instead of a smooth one, one obtains, a priori, a different geodesic dimension $\mathcal{N}(x)$.
\end{rmk}
\begin{rmk}
The positivity assumption on $\mu$ is essential to describe the equality case. For example, if $X = \mathbb{R}$ with the Euclidean metric and $\mu= x^2 dx$, we have $\mathcal{N}(x) = 1$ for $x \neq 0$ and $\mathcal{N}(x) =3$ for $x =0$. Clearly $dx$ and $x^2 dx$ are not commensurable.
\end{rmk}

\subsection{A lower bound for the MCP dimension}

If $(X,d,\mu)$ satisfies the $\mathrm{MCP}(K,N)$, then $N \geq \mathcal{N}(x)$ at any point. We give here a general statement for metric measure spaces (which, we recall, are always assumed to have negligible cut loci).
\begin{theorem}\label{t:bound}
Let $(X,d,\mu)$ be a metric measure space, with geodesic dimension $\mathcal{N}(x)$, that satisfies the $\mathrm{MCP}(K,N)$, for some $K \in \R$ and $N >1$ or $K\leq 0$ and $N=1$. Then
\begin{equation}
N \geq \sup \{\mathcal{N}(x) \mid x \in X\}.
\end{equation}
\end{theorem}
The following definition was given originally in \cite{Rifford} for Carnot groups.
\begin{definition}
Let $(X,d,\mu)$ be a metric measure space that satisfies the $\mathrm{MCP}(0,N)$ for some $N\geq 1$. Its \emph{curvature exponent} is
\begin{equation}
N_0:= \inf\{N>1 \mid \mathrm{MCP}(0,N) \text{ is satisfied}\}.
\end{equation}
When $(X,d,\mu)$ does not satisfy the $\mathrm{MCP}(0,N)$ for all $N\geq 1$, we set $N_0 = +\infty$.
\end{definition}
If $N_0 < + \infty$, then the $\mathrm{MCP}(0,N_0)$ is satisfied. Theorem~\ref{t:bound} implies that $N_0 \geq \mathcal{N}$. It may happen that $N_0 > \mathcal{N}$ strictly, as in the following example.
\begin{example}[Riemannian Heisenberg]\label{ex:ex1}
Consider the Riemannian structure generated by the following global orthonormal vector fields, in coordinates $(x,y,z) \in \R^3$:
\begin{equation}
X = \partial_x -\frac{y}{2}\partial_z, \qquad Y = \partial_y + \frac{x}{2}\partial_z, \qquad Z = \partial_z.
\end{equation}
Being a Riemannian structure, $\mathcal{N} = 3$. In \cite{Rifford} it is proved that, when equipped with the Riemannian volume, it satisfies the $\mathrm{MCP}(0,5)$. With the same computations it is easy to prove that the $\mathrm{MCP}(0,5-\varepsilon)$ is violated for any $\varepsilon>0$, so its curvature exponent is $N_0 = 5$.
\end{example}
%

\subsection{Back to Carnot groups}

In \cite{Rifford} Rifford studied the measure contraction properties of general Carnot groups. It may happen that $N_0 = +\infty$, that is the $\mathrm{MCP}(0,N)$ is never satisfied. However, if the Carnot group is ideal (i.e. it does not admit non-trivial abnormal minimizing curves), we have the following result.
\begin{theorem}[Rifford \cite{Rifford}]
Let $(G,d,\mu)$ be a Carnot group. Assume it is ideal. Then it satisfies the $\mathrm{MCP}(0,N)$ for some $N > 1$. In particular its curvature exponent $N_0$ is finite.
\end{theorem}
The proof of the above result is based on a semiconcavity property of the distance for ideal structures, which does not hold in general. Nevertheless, Theorem~\ref{t:corank1} shows that the above statement can hold even in presence of non-trivial abnormal minimizers. In general, nothing is known on the finiteness of $N_0$, but we have the following lower bound.
\begin{theorem}[Rifford \cite{Rifford}]
Let $(G,d,\mu)$ be a Carnot group. Assume it is geodesic with negligible cut loci. Then its curvature exponent $N_0$ satisfies
\begin{equation}\label{eq:boundRiff}
N_0 \geq N_R:=Q+n-k,
\end{equation}
where $Q$ is the Hausdorff dimension, $n$ is the topological one, and $k$ is the rank of the horizontal distribution.
\end{theorem}
For Carnot groups, the geodesic dimension $\mathcal{N}(x)=\mathcal{N}$ is clearly constant. In particular $N_0 \geq \mathcal{N}$, by Theorem~\ref{t:bound}. This lower bound improves \eqref{eq:boundRiff}, as a consequence of the following.
\begin{theorem}\label{p:fat}
A Carnot group is ideal if and only if it is fat\footnote{A sub-Riemannian structure $(M,\distr,g)$ is \emph{fat} if for all $x \in M$ and $X \in \distr$, $X(x) \neq0$, then $\distr_x + [X,\distr]_x=T_xM$. It is \emph{ideal} if it is complete and does not admit non-trivial abnormal minimizers.}. In this case, $\mathcal{N} = N_R$. If a Carnot group has step $s>2$, then  $\mathcal{N} > N_R$.
\end{theorem}
\begin{rmk}
Since fat Carnot groups do not admit non-trivial abnormal curves, the first part of Theorem~\ref{p:fat} can be restated as follows: a Carnot group admits a non-trivial abnormal curve if and only if it admits a non-trivial abnormal minimizer  (see Section~\ref{s:preliminaries}).
\end{rmk}
\begin{example}[Engel group]\label{ex:engel}
Consider the Carnot group in dimension $4$, generated by the following global orthonormal left-invariant vector fields in coordinates $(x_1,x_2,x_3,x_4) \in \R^4$
\begin{equation}
X_1 = \partial_1, \qquad X_2 = \partial_2 + x_1 \partial_3 + x_1 x_2 \partial_4.
\end{equation}
The Engel group is a metric space with negligible cut loci (see Remark~\ref{r:Engelcut}). It has rank $2$, step $3$, dimension $4$ and growth vector $(2,3,4)$. Its Hausdorff dimension is $Q = 7$. The geodesic dimension is $\mathcal{N} = 10$ (see Section~\ref{s:Engel}), while $N_R = 9$. This is the lowest dimensional Carnot group where $\mathcal{N}>N_R$.

Checking whether the Engel group satisfies the $\mathrm{MCP}(0,\mathcal{N})$ should be possible, at least in principle, as expressions for the Jacobian determinant are known \cite{AS-Engel}.
\end{example}

\subsection{Open problems}
As a consequence of the formula for $\mathcal{N}(x)$ in the (sub-)Riemannian setting (see Section~\ref{s:geod}), for any corank $1$ Carnot group we have
\begin{equation}
\mathcal{N} = k+3.
\end{equation}
Thus, Theorem~\ref{t:corank1} can be restated saying that \emph{for any corank $1$ Carnot group, the curvature exponent is equal to the geodesic dimension}. Moreover, for $\mathbb{H}_{2d+1}$, this gives $\mathcal{N} =2d+3$, and coincides with the ``mysterious'' integer originally found by Juillet.

The class of corank 1 Carnot groups includes non-ideal structures (see Remark~\ref{r:ideal}). We do not know whether other non-ideal Carnot groups enjoy some $\mathrm{MCP}(0,N)$. It is not even known whether general Carnot groups have negligible cut loci (this is related with the Sard conjecture in sub-Riemannian geometry \cite{RT-MorseSard,DMOPV-Sard}). However, if they do, it is natural to expect the curvature exponent to be equal to the curvature dimension. 

\begin{conjecture}
Let $(X,d,\mu)$ be a Carnot group. Assume that it has negligible cut loci. Then the geodesic dimension coincides with the curvature exponent.
\end{conjecture}

Preliminary results (using sub-Riemannian curvature techniques, in collaboration with D. Barilari) seem to provide evidence to the above claim for some step 2 Carnot groups.

\subsection*{Structure of the paper}
In Section~\ref{s:preliminaries} we collect some preliminaries of sub-Riemannian geometry and Carnot groups. In Section~\ref{s:corank1} we characterize the minimizers of corank $1$ Carnot groups. In Section~\ref{s:proofcorank1}, \ref{s:proofbasicprop}, \ref{s:proofbound} we prove Theorems~\ref{t:corank1}, \ref{t:basicprop}, \ref{t:bound} respectively. In Section~\ref{s:geod} we recall the formula for the geodesic dimension on general sub-Riemannian structures, we prove Theorem~\ref{t:basicprop2} and we discuss the Engel example. In Section~\ref{s:fat} we prove Theorem~\ref{p:fat}.

\subsection*{Acknowledgments}
I warmly thank D. Barilari for many fruitful discussions, and the anonymous referee for many useful comments.
This research was supported by the ERC StG 2009 ``GeCoMethods'', contract n. 239748, by the iCODE institute (research project of the Idex Paris-Saclay), and by the ANR project ``SRGI'' ANR-15-CE40-0018. This research benefited from the support of the ``FMJH Program Gaspard Monge in optimization and operation research'' and from the support to this program from EDF.

\section{Sub-Riemannian geometry} \label{s:preliminaries}
We present some basic results in sub-Riemannian geometry. See \cite{nostrolibro,rifford2014sub,montgomerybook} for reference.

\subsection{Basic definitions} A sub-Rieman\-nian manifold is a triple $(M,\distr,g)$, where $M$ is a smooth, connected manifold of dimension $n \geq 3$, $\distr$ is a vector distribution of constant rank $k \leq n$ and $g$ is a smooth metric on $\distr$. We always assume that the distribution is bracket-generating. A \emph{horizontal curve} $\gamma : [0,1] \to M$ is a Lipschitz continuous path such that $\dot\gamma(t) \in \distr_{\gamma(t)}$ for almost any $t$. Horizontal curves have a well defined \emph{length}
\begin{equation}
\ell(\gamma) = \int_0^1 \sqrt{g(\dot\gamma(t),\dot\gamma(t))}dt.
\end{equation}
The \emph{sub-Rieman\-nian (or Carnot-Carath\'eodory) distance} is defined by:
\begin{equation}
d(x,y) = \inf\{\ell(\gamma)\mid \gamma(0) = x,\, \gamma(1) = y,\, \gamma \text{ horizontal} \}.
\end{equation}
By the Chow-Rashevskii theorem, under the bracket-generating condition, $d: M \times M \to \R$ is finite and continuous. A sub-Rieman\-nian manifold is complete if $(M,d)$ is complete as a metric space. In this case, for any $x,y \in M$ there exists a minimizing geodesic joining the two points. In place of the length $\ell$, one can consider the \emph{energy functional} as
\begin{equation}
J(\gamma) = \frac{1}{2}\int_0^1 g(\dot\gamma(t),\dot\gamma(t)) dt.
\end{equation}
It is well known that, on the space of horizontal curves with fixed endpoints, the minimizers of $J(\cdot)$ coincide with the minimizers of $\ell(\cdot)$ with constant speed. Since $\ell$ is invariant by reparametrization (and in particular we can always reparametrize horizontal curves in such a way that they have constant speed), we do not loose generality in defining \emph{geodesics} as horizontal curves that are locally energy minimizers between their endpoints.

\subsection{Hamiltonian}

We define the \emph{Hamiltonian function} $ H : T^*M \to \R$ as
\begin{equation}
H(\lambda) = \frac{1}{2}\sum_{i=1}^k \langle \lambda,X_i\rangle, \qquad \lambda \in T^*M,
\end{equation}
for any local orthonormal frame $X_1,\ldots,X_k$ for $\distr$. Here $\langle \lambda,\cdot\rangle$ denotes the dual action of covectors on vectors. The cotangent bundle $\pi:T^*M \to M$ is equipped with a natural symplectic form $\sigma$. The Hamiltonian vector field $\vec{H}$ is the unique vector field such that $\sigma(\cdot,\vec{H}) = dH$. In particular, the \emph{Hamilton equations} are
\begin{equation}\label{eq:Ham}
\dot\lambda(t) = \vec{H}(\lambda(t)), \qquad \lambda(t) \in T^*M.
\end{equation}
If $(M,d)$ is complete, any solution of \eqref{eq:Ham} can be extended to a smooth curve for all times.

\subsection{End-point map}

Let $\gamma_u :[0,1] \to M$ be an horizontal curve joining $x$ and $y$. Up to restriction and reparametrization, we assume that the curve has no self-intersections. Thus we can find a smooth orthonormal frame $X_1,\ldots,X_k$ of horizontal vectors fields, defined in a neighborhood of $\gamma_u$. Moreover, there is a \emph{control} $u \in L^\infty([0,1],\R^k)$ such that
\begin{equation}
\dot\gamma_u(t) =  \sum_{i=1}^k u_i(t) X_i(\gamma_u(t)), \qquad \text{a.e. } t \in [0,1].
\end{equation}
Let $\mathcal{U} \subset L^\infty([0,1],\R^k)$ be the open set such that, for $v \in \mathcal{U}$, the solution of
\begin{equation}
\dot\gamma_v(t) = \sum_{i=1}^k v_i(t) X_i(\gamma_v(t)), \qquad \gamma_v(0) = x,
\end{equation}
is well defined for a.e. $t \in [0,1]$. Clearly $u \in \mathcal{U}$. We define the \emph{end-point map} with base $x$ as $E_{x}: \mathcal{U} \to M$ that sends $v$ to $\gamma_v(1)$. The end-point map is smooth on $\mathcal{U}$.

\subsection{Lagrange multipliers}

We can see $J : \mathcal{U} \to \R$ as a smooth functional on $\mathcal{U}$ (we are identifying $\mathcal{U}$ with a neighborhood of $\gamma_u$ in the space of horizontal curves starting from $x$). A minimizing geodesic $\gamma_u$ is a solution of the constrained minimum problem
\begin{equation}
J(v) \to \text{min}, \qquad E_x(v) = y, \qquad v \in \mathcal{U}.
\end{equation}
By the Lagrange multipliers rule, there exists a non-trivial pair $(\lambda_1,\nu)$, such that
\begin{equation}\label{eq:multipliers}
\lambda_1 \circ D_u E_x  = \nu D_u J, \qquad \lambda_1 \in T_y^*M, \qquad\nu \in \{0,1\},
\end{equation}
where $\circ$ denotes the composition and $D$ the (Fr\'echet) differential. If $\gamma_u : [0,1] \to  M$ with control $u \in \mathcal{U}$ is an horizontal curve (not necessarily minimizing), we say that a non-zero pair $(\lambda_1,\nu) \in T_y^*M \times \{0,1\}$ is a \emph{Lagrange multiplier} for $\gamma_u$ if \eqref{eq:multipliers} is satisfied. The multiplier $(\lambda_1,\nu)$ and the associated curve $\gamma_u$ are called \emph{normal} if $\nu = 1$ and \emph{abnormal} if $\nu = 0$. Observe that Lagrange multipliers are not unique, and a horizontal curve may be both normal \emph{and} abnormal. Observe also that $\gamma_u$ is an abnormal curve if and only if $u$ is a critical point for $E_x$. In this case, $\gamma_u$ is also called a \emph{singular curve}. The following characterization is a consequence of the Pontryagin Maximum Principle \cite{Agrabook}.
\begin{theorem}
Let $\gamma_u :[0,1] \to M$ be an horizontal curve joining $x$ with $y$. A non-zero pair $(\lambda_1,\nu) \in T_y^*M \times \{0,1\}$ is a Lagrange multiplier for $\gamma_u$ if and only if there exists a Lipschitz curve $\lambda(t) \in T_{\gamma_u(t)}^*M$ with $\lambda(1) = \lambda_1,$ such that
\begin{itemize}
\item if $\nu = 1$ then $\dot{\lambda}(t) = \vec{H}(\gamma(t))$, i.e.\ it is a solution of Hamilton equations,
\item if $\nu =0 $ then $\sigma(\dot\lambda(t), T_{\lambda(t)} \distr^\perp) = 0$,
\end{itemize}
where $\distr^\perp \subset T^*M$ is the sub-bundle of covectors that annihilate the distribution.
\end{theorem}
In the first (resp. second) case, $\lambda(t)$ is called a \emph{normal} (resp. \emph{abnormal}) \emph{extremal}. Normal extremals are   integral curves $\lambda(t)$ of $\vec{H}$. As such, they are smooth, and characterized by their \emph{initial covector} $\lambda = \lambda(0)$. A geodesic is normal (resp. abnormal) if admits a normal (resp. abnormal) extremal. On the other hand, it is well known that the projection $\gamma_\lambda(t) = \pi(\lambda(t))$ of a normal extremal is locally minimizing, hence it is a normal geodesic (see \cite[Chapter 4]{nostrolibro} or \cite[Theorem 1.5.7]{montgomerybook}). The \emph{exponential map} at $x \in M$ is the map
\begin{equation}
\exp_x : T_x^*M \to M,
\end{equation}
which assigns to $\lambda \in T_x^*M$ the final point $\pi(\lambda(1))$ of the corresponding normal geodesic. The curve $\gamma_\lambda(t):=\exp_x(t \lambda)$, for $t \in [0,1]$, is the normal geodesic corresponding to $\lambda$, which has constant speed $\|\dot\gamma_\lambda(t)\| = \sqrt{2H(\lambda)}$ and length $\ell(\gamma|_{[t_1,t_2]}) = \sqrt{2H(\lambda)}(t_2-t_1)$.

\begin{definition}\label{d:ideal}
A sub-Riemannian structure $(M,\distr,g)$ is \emph{ideal} if it is complete and does not admit non-trivial abnormal minimizers.
\end{definition}
\begin{definition}\label{d:fat}
A sub-Riemannian structure $(M,\distr,g)$ is \emph{fat} (or \emph{strong bracket-generating}) if for all $x \in M$ and $X \in \distr$, $X(x) \neq0$, then $\distr_x + [X,\distr]_x=T_xM$. 
\end{definition}
The definition of ideal structures appears in \cite{Rifford,rifford2014sub}, in the equivalent language of singular curves. We stress that fat sub-Riemannian structures admit no non-trivial abnormal curves (see \cite[Section 5.6]{montgomerybook}). In particular, complete fat structures are ideal.

\subsection{Carnot groups}
A \emph{Carnot group $(G,\star)$ of step $s$} is a connected, simply connected Lie group of dimension $n$, such that its Lie algebra $\mathfrak{g} = T_e G$ is stratified of step $s$, that is
\begin{equation}
\mathfrak{g} = \mathfrak{g}_1 \oplus \ldots \oplus \mathfrak{g}_s,
\end{equation}
with
\begin{equation}
[\mathfrak{g}_1,\mathfrak{g}_j] = \mathfrak{g}_{1+j},\quad \forall 1\leq j\leq s, \quad \mathfrak{g}_s \neq \{0\}, \quad \mathfrak{g}_{s+1} =\{0\}.
\end{equation}
The \emph{group exponential map} $\mathrm{exp}_{G} : \mathfrak{g} \to G$ associates with $V \in \mathfrak{g}$ the element $\gamma_V(1)$, where $\gamma_V: [0,1] \to G$ is the integral line, starting at $\gamma_V(0)=e$, of the left invariant vector field associated with $V$. Since $G$ is simply connected and $\mathfrak{g}$ is nilpotent, $\mathrm{exp}_G$ is a smooth diffeomorphism. Thus, the choice of a basis of $\mathfrak{g}$ induces coordinates on $G \simeq \R^{n}$, which are called \emph{exponential coordinates}.

Let $\distr$ be the left-invariant distribution generated by $\mathfrak{g}_1$, with a left-invariant scalar product $g$.  This defines a sub-Riemannian structure $(G,\distr,g)$ on the Carnot group. For $x \in G$, we denote with $L_x(y) := x \star y$ the left translation. The map $L_x : G \to G$ is a smooth isometry.  Any Carnot group, equipped with the Carnot-Carath\'eodory distance $d$ and the Lebesgue measure $\mu$ of $G=\R^n$ is a complete metric measure space $(X,d,\mu)$. Haar, Popp, Lebesgue and the top-dimensional Hausdorff measures are left-invariant and proportional.

\section{Corank 1 Carnot groups}\label{s:corank1}

A corank $1$ Carnot group is a Carnot groups of step $s=2$, with $\dim\mathfrak{g}_1 = k$ and $\dim \mathfrak{g}_2 = 1$. In exponential coordinates $(x,z)$ on $\R^k \times \R$, they are generated by the following set of global orthonormal left-invariant frames
\begin{equation}
X_i = \partial_{x_i} - \frac{1}{2}\sum_{j=1}^{k} A_{ij} x_j \partial_z, \qquad i =1,\ldots, k,
\end{equation}
where $A$ is a $k\times k$ skew symmetric matrix.  Observe that
\begin{equation}
[X_i,X_j] = A_{ij} \partial_z, \qquad  i,j = 1,\ldots,k.
\end{equation}
Let $0<\alpha_1\leq \ldots \leq \alpha_d$ be the non-zero singular values of $A$. In particular, $\dim \ker A = k-2d$. Up to an orthogonal change of coordinates, we can assume that
\begin{equation}
A = \begin{pmatrix}
\mathbbold{0}	&			&			\\
& \alpha_1 J 	& 			& 		 \\
& 			& \ddots 	&		 \\
& 			&			& \alpha_d J  \\
\end{pmatrix}, \qquad J = \begin{pmatrix}
0 & 1 \\
-1 & 0
\end{pmatrix}.
\end{equation}
The first zero block has dimension $k-2d$, while each other diagonal block is $2 \times 2$. We split the coordinate $x = (x^0,x^1,\ldots,x^d)$, where $x^0 \in \R^{k-2d}$ and $x^i \in \R^2$, for $i=1,\ldots,d$. 

If $A$ has trivial kernel (in particular, $k$ is even), we are in the case of a \emph{contact Carnot group}, and there are no non-trivial abnormal minimizers. However, when $A$ has a non-trivial kernel, then non-trivial abnormal minimizers appear. To prove Theorem~\ref{t:bound}, we need a complete characterization of the minimizing geodesics on a general corank $1$ Carnot group. We extend the results of \cite{ABB-Hausdorff}, where the case of a non-degenerate $A$ is considered.

\subsection{Characterization of minimizers}

On any corank $1$ sub-Riemannian distribution, all minimizing geodesics are normal (this is true for any step $2$ distribution). In particular, they can be recovered by solving Hamilton equations. By left-invariance, it is sufficient to consider geodesics starting from the identity $e = (0,0)$. Any covector $\lambda \in T^*_eG$ has coordinates $(p_x,p_z)$, where we split $p_x = (p_x^0,p_x^1,\ldots,p_x^d)$.

\begin{lemma}\label{l:geods}
The exponential map $\exp_e :T^*_e G \to G$ of a Corank 1 Carnot group is
\begin{equation}
\exp_e(p_x^0,p_x^1,\ldots,p_x^d,p_z) = (x^0,x^1,\ldots,x^d,z),
\end{equation}
where, for all $i=1,\ldots,d$ we have
\begin{align}
x^0 & = p_x^0, \\
x^i & = \left(\frac{\sin(\alpha_i p_z)}{\alpha_i p_z}I + \frac{\cos(\alpha_i p_z )-1}{\alpha_i p_z} J\right) p_x, \\
z & =\sum_{i=1}^d \|p_x^i\|^2\frac{\alpha_i p_z - \sin(\alpha_i p_z)}{2 \alpha_i p_z^2}.
\end{align}
If $p_z = 0$, one must consider the limit $p_z \to 0$, that is $\exp_e(p_x,0) = (p_x,0)$.
\end{lemma}
\begin{rmk}[Abnormal geodesics]\label{r:abnormal}
A non-zero covector $\lambda = (p_x,p_z)$ such that $Ap_x = 0$, that is of the form $(p_x^0,0,\ldots,0,p_z)$ corresponds to an abnormal geodesic. A way to see this is to observe that there is an infinite number of initial covectors giving the same geodesic
\begin{equation}
\exp_e(t p_x^0,0,\ldots,0,t p_z) = (t p_x^0,0, \ldots,0,0), \qquad \forall p_z \in \R.
\end{equation}
A direct analysis of the end-point map shows that abnormal geodesic are all of this type.
\end{rmk}
\begin{proof}
Let $h_x = (h_1,\ldots,h_{k}) : T^*G \to \R^k$ and $h_z : T^*G \to \R$, where $h_i(\lambda):= \langle \lambda, X_i\rangle$, for $i=1,\ldots,k$ and $h_z(\lambda) := \langle \lambda, \partial_z\rangle$. Thus, $H = \frac{1}{2}\|h_x\|^2$. Hamilton equations are
\begin{equation}
\dot{h}_z = 0, \qquad
\dot{h}_x = -h_z A h_x, \qquad \dot{x} = h_x, \qquad \dot{z} = -\tfrac{1}{2}h_x^* A x,
\end{equation}
where, without risk of confusion, the dot denotes the derivative with respect to $t$. We have
\begin{equation}
h_z(t) = p_z, \qquad h_x(t) = e^{-p_z A t} p_x.
\end{equation}
The equations for $(x,z)$ can be easily integrated, using the block-diagonal structure of $A$. Split $h_x = (h_x^0,h_x^1,\ldots,h_x^d)$, with $h_x^0 \in \R^{k-2d}$ and $h_x^i \in \R^2$ for $i=1,\ldots,d$. We obtain
\begin{equation}
h_x^0(t) = p_x^0, \qquad h_x^i(t) = [\cos(\alpha_i p_z t) I - \sin(\alpha_i p_z t) J]p_x^i,
\end{equation}
where $I$ is the $2\times 2$ identity matrix. Integrating the above on $[0,t]$, we obtain
\begin{equation}
x^0(t) = p_x^0 t, \qquad x^i(t) = \left(\frac{\sin(\alpha_i p_z t)}{\alpha_i p_z} I + \frac{\cos(\alpha_i p_z t) - 1}{\alpha_i p_z} J\right)p_x^i.
\end{equation}
Finally, for the coordinate $z$ we obtain
\begin{align}
z & = -\frac{1}{2}\int_0^1 h_x^*(s) A x(s) ds  =  -\frac{1}{2} \sum_{i=1}^d \int_0^1 h_x^i(s)^* \alpha_i J x^i(s) ds \\
& =\frac{1}{2p_z}\sum_{i=1}^d \|p_x^i\|^2 \int_0^1\left(1 - \cos(\alpha_i p_z s)\right)ds =\sum_{i=1}^d \|p_x^i\|^2 \left(\frac{\alpha_i p_z - \sin(\alpha_i p_z)}{2\alpha_i p_z^2}\right). \qedhere
\end{align}
\end{proof}

\begin{lemma}\label{l:jacobian}
The Jacobian determinant of the exponential map is
\begin{equation}
J(p_x,p_z) = \frac{2^{2d}}{\alpha^2 p_z^{2d+2}}\sum_{i=1}^d \|p_x^i\|^2 \prod_{j\neq i} \sin\left(\frac{\alpha_j p_z}{2}\right)^2 \sin\left(\frac{\alpha_i p_z}{2}\right)\left(\sin\left(\frac{\alpha_i p_z}{2}\right)-\frac{\alpha_i p_z}{2}\cos\left(\frac{\alpha_i p_z}{2}\right)\right),
\end{equation}
where $\alpha = \prod_{i=1}^d \alpha_i$ is the product of the non-zero singular values of $A$. If $p_z = 0$, the formula must be taken in the limit $p_z \to 0$. In particular $J(p_x,0)= \frac{1}{12} \sum_{i=1}^d \|p_x^i\|^2 \alpha_i^2$.
\end{lemma}
\begin{proof}
For any matrix with the following block structure
\begin{equation}
M = \begin{pmatrix}
B & v \\
w^* & \theta
\end{pmatrix},
\end{equation}
where the only constraint is that $\theta \in \R$ is a one-dimensional block, we have 
\begin{equation}
\det(M) = \theta \det(B) - v^* \mathrm{cof}(B) w,
\end{equation}
where $\mathrm{cof}$ denotes the matrix of cofactors. More in general, let
\begin{equation}
M = \begin{pmatrix}
B_0 & & & &  v_0 \\
 & B_1 & & & v_1 \\
 & & \ddots & & \vdots \\
 & & & B_d & v_d \\
w_0^* & w_1^* & \dots & w_d^* & \theta
\end{pmatrix},
\end{equation}
where $B_0,\ldots,B_d$ are square blocks of arbitrary (possibly different) dimension, $\theta \in \R$ and $v_i,w_j$ are column vectors of the appropriate dimension. In this case we have
\begin{equation}
\det(M) = \theta \prod_{i=0}^d \det B_i - \sum_{i=0}^d \left(\prod_{j \neq i} \det B_j\right) v_i^* \mathrm{cof}(B_i) w_i.
\end{equation}
If $B_i = a_i I + b_i J$, then $\mathrm{cof}(B_i) = B_i$. If we also assume that $B_0 =1$, $v_0 = w_0 = 0$, we have
\begin{equation}\label{eq:formuladet}
\det(M) = \theta \prod_{i=1}^d \det B_i - \sum_{i=1}^d \left(\prod_{j \neq i} \det B_j\right) v_i^* B_i w_i.
\end{equation}
From Lemma~\ref{l:geods}, the differential of the exponential map has the above form, with
\begin{align}
B_0 & = \frac{\partial x^0}{\partial p_x^0} = 1, \qquad B_i  = \frac{\partial x^i}{\partial p_x^i}  = \frac{\sin(\alpha_i p_z)}{\alpha_i p_z}I + \frac{\cos(\alpha_i p_z )-1}{\alpha_i p_z} J, \\
v_i & = \frac{\partial x^i}{\partial p_z} = \frac{\alpha_i p_z \cos(\alpha_i p_z) - \sin(\alpha_i p_z)}{\alpha_i p_z^2} I p_x^i + \frac{1-\cos(\alpha_i p_z) - \alpha_i p_z \sin(\alpha_i p_z)}{\alpha_i p_z^2} J p_x^i, \\
w_i & = \frac{\partial z}{\partial p_x^i} = \frac{\alpha_i p_z - \sin (\alpha_i p_z)}{\alpha_i p_z^2} p_x^i, \\ 
\theta  &= \frac{\partial z}{\partial p_z} = \sum_{i=1}^d \|p_x^i\|^2 \left( \frac{2 \sin(\alpha_i p_z) - \alpha_i p_z - \alpha_i p_z \cos(\alpha_i p_z)}{2\alpha_i p_z^3}\right) .
\end{align}
The result follows applying formula \eqref{eq:formuladet} and observing that, for $i=1,\ldots,d$, we have
\begin{align}
\det(B_i) & = \frac{4 \sin^2(\alpha_i p_z/2)}{(\alpha_i p_z)^2}, \\
v_i^* B_i w_i & =\frac{\alpha_i^2 \|p_x^i\|^2}{(\alpha_i p_z)^5} (\alpha_i p_z-\sin (\alpha_i p_z)) (\alpha_i p_z \sin (\alpha_i p_z)+2 \cos (\alpha_i p_z)-2) . \qedhere
\end{align}
\end{proof}

\begin{lemma}[Characterization of the cotangent injectivity domain]\label{l:ingjd}
Consider the set
\begin{equation}
D := \left\lbrace \lambda = (p_x,p_z) \in T_e^*G \text{ such that } |p_z| < \frac{2\pi}{\alpha_d} \text{ and } A p_x \neq 0 \right\rbrace \subset T_e^*G.
\end{equation}
Then $\exp_e : D \to \exp_e(D)$ is a smooth diffeomorphism and $\mathcal{C}(e):=G \setminus \exp_e(D) $ is a closed set with zero measure. 
\end{lemma}
\begin{proof}
Since all geodesic are normal and $(G,d)$ is complete, each point of $G$ is reached by at least one minimizing normal geodesic $\gamma_\lambda :[0,1] \to G$, with $\lambda =(p_x,p_z) \in T_e^*G$. If $|p_z| > 2 \pi/ \alpha_d$ and $A p_x \neq 0$, then $\gamma_\lambda$ is a strictly normal geodesic (i.e. not abnormal) with a conjugate time at $t_* = 2 \pi/\alpha_d |p_z| < 1$. Strictly normal geodesics lose optimality after their first conjugate time (see \cite{nostrolibro}), hence $\gamma_\lambda(t)$ is not minimizing on $[0,1]$. On the other hand, if $A p_x =0$, for any value of $p_z$ we obtain the same abnormal geodesic (see Remark~\ref{r:abnormal}). It follows that $\exp_e : \bar{D} \to G$ is onto (the bar denotes the closure). Thus, $\exp_e: D \to \exp_e(D)$ is onto and $\mathcal{C}(e) = G \setminus \exp_e(D) = \exp_e(\partial D)$ has zero measure.

We now prove that, if $\lambda \in D$, then $\gamma_\lambda :[0,1] \to G$ is the unique geodesic joining its endpoints. In fact, assume that there are two covectors $\lambda=(p_x,p_z)$ and $\bar\lambda=(\bar p_x,\bar p_z) \in D$, such that $\exp_e(\lambda) = \exp_e(\bar \lambda)$. Since the two geodesics have the same length, $\|p_x\| = \ell(\gamma_\lambda) =\ell(\gamma_{\bar\lambda}) = \|\bar p_x\|$. Using Lemma~\ref{l:geods}, we have $p_x^0 = \bar p_x^0$ and
\begin{equation}\label{eq:sinc}
\|x^i\|^2 = 4 \|p_x^i\|^2 \sinc(\alpha_i p_z)^2 =4 \|\bar p_x^i\|^2 \sinc(\alpha_i \bar p_z)^2 , \qquad \forall i=1,\ldots,d,
\end{equation}
where $\sinc(w) = \sin(w)/w$ is positive and strictly decreasing on $[0,\pi)$. Since $A p_x , A \bar p_x \neq 0$, there exist two non-empty set of indices $I,\bar I \subset \{1,\ldots,d\}$ such that, for $i \in I$ (resp. $\bar I$) we have $\|p_x^i\|^2 \neq 0$ (resp. $\|\bar p_x^i\|^2 \neq 0$). Since $\alpha_i p_z, \alpha_i \bar p_z < \pi$, by \eqref{eq:sinc}, we have $I= I'$.

Assume now that $\bar p_z > p_z$. Then by \eqref{eq:sinc} $\|\bar p_x^i\|^2 > \|p_x^i\|^2$ for all $i \in I$. In particular 
\begin{equation}
\|\bar p_x\|^2 =\|\bar p_x^0\|^2+ \sum_{i \in I} \|\bar p_x^i\|^2 >  \|p_x^0\|^2+ \sum_{i \in I} \|p_x^i\|^2 = \|p_x\|^2,
\end{equation}
which is a contradiction. Analogously if $\bar p_z < p_z$, with reversed inequalities. Thus $p_z = \bar p_z$. Using now the equations for the coordinate $x^i$ of Lemma~\ref{l:geods} we observe that
\begin{equation}
\left[\frac{\sin(\alpha_i p_z)}{\alpha_i p_z} I + \frac{\cos(\alpha_i p_z)-1}{\alpha_i p_z} J\right](\bar p_x^i - p_x^i) = 0, \qquad \forall i = 1, \ldots,d.
\end{equation}
The $2 \times 2$ matrix on the left hand side is invertible (since if $\alpha_i p_z < \pi$), hence also $\bar p_x = p_x$. Thus $\exp_e :D \to \exp_e(D)$ is invertible.

Finally, no point $\lambda \in D$ can be critical for $\exp_e$. In fact, from Lemma~\ref{l:jacobian} we have that $J(p_x,p_z) = \sum_{i=1}^d \|p_x^i\|^2 f_i(p_z)$, where each $f_i(p_z) > 0$ for $p_z < 2 \pi/ \alpha_d$. In particular $J(p_x,p_z) = 0$ if and only if $Ap_x = 0$. But this closed set was excluded from $D$.
\end{proof}

\begin{cor}\label{c:good}
For any $x \in G$, let $\mathcal{C}(x):= L_x \mathcal{C}(e)$, where $L_x :G \to G$ is the left-translation. There exists a measurable map $\Phi^x : G \setminus \mathcal{C}(x) \times [0,1] \to G$, given by
\begin{equation}
\Phi^x(y,t) = L_x \exp_{e}(t \exp_e^{-1}(L_x^{-1} y)),
\end{equation}
such that $\Phi^x(y,t)$ is the unique minimizing geodesic joining $x$ with $y$.
\end{cor}
The next key lemma and its proof are a simplified version of the original concavity argument of Juillet for the Heisenberg group \cite[Lemma 2.6]{Juillet}.

\begin{lemma}\label{l:ineq}
Let $g(x):= \sin(x) - x \cos(x)$. Then, for all $x \in (0,\pi)$ and $t \in [0,1]$,
\begin{equation}
g(t x) \geq t^N g(x), \qquad \forall N \geq 3.
\end{equation}
\end{lemma}
\begin{proof}
The condition $N \geq 3$ is necessary, as $g(x) = x^3/3 + O(x^4)$. It is sufficient to prove the statement for $N=3$. The cases $t=0$ and $t=1$ are trivial, hence we assume $t \in (0,1)$. By Gronwall's Lemma the above statement is implied by the differential inequality
\begin{equation}
g'(s) \leq  3 g(s)/s, \qquad s \in (0,\pi).
\end{equation}
In fact, it is sufficient to integrate the above inequality on $[tx,x] \subset (0,\pi)$ to prove our claim. The above inequality reads
\begin{equation}
f(s):=(3-s^2) \sin (s)-3 s \cos (s) \geq 0, \qquad s \in (0, \pi).
\end{equation}
To prove it, we observe that $f(0) = 0$ and $f'(s) =  s (\sin (s)-s \cos (s)) \geq 0$ on $(0,\pi)$.
\end{proof}
\begin{cor}\label{c:jacobineq}
For all $(p_x,p_z) \in D$, we have the following inequality
\begin{equation}
\frac{J(tp_x, tp_z)}{J(p_x,p_z)} \geq t^2, \qquad \forall t \in [0,1].
\end{equation}
\end{cor}
\begin{proof}
Apply Lemma \ref{l:ineq} to the explicit expression of $J$ from Lemma~\ref{l:jacobian}, and then use the standard inequality $\sin(t x) \geq t \sin(x)$, valid for all $x \in [0,\pi]$ and $t \in [0,1]$.
\end{proof}

\section{Proof of Theorem \ref{t:corank1}}\label{s:proofcorank1}

The proof combines the arguments of \cite{Juillet} and the computation of the Jacobian determinant of \cite{ABB-Hausdorff} for contact Carnot groups, extended here to the general corank $1$ case. 

\subsection{Step 1}
We first prove that the $\mathrm{MCP}(0,N)$ holds for $N\geq k+3$. By left-translation, it is sufficient to prove the inequality \eqref{eq:MCP} for the homothety with center equal to the identity $e=(0,0)$. Let $\Omega$ be a measurable set with $0<\mu(\Omega)<+\infty$. 

By Lemma~\ref{l:ingjd}, up to removing a set of zero measure, $\Omega = \exp_e(A)$ for some $A \subset D \subset T_e^*G$. On the other hand, by Corollary~\ref{c:good}, we have
\begin{equation}
\Omega_t = \exp_e(t A), \qquad \forall t \in [0,1],
\end{equation}
where $t A $ denotes the set obtained by multiplying by $t$ any point of the set $A \subset T_e^*G$ (an Euclidean homothety). Thus, for all $t \in [0,1]$ we have
\begin{align}
\mu(\Omega_t) & = \int_{\Omega_t} d \mu = \int_{t A} J(p_x,p_z) dp_x dp_z \\
& = t^{k+1}  \int_{A} J(t p_z, tp_z) dp_x dp_z \geq  t^{k+3}  \int_{A} J(p_x,p_z) dp_x dp_z = t^{k+3} \mu(\Omega),
\end{align}
where we used Corollary~\ref{c:jacobineq}. In particular $\mu(\Omega_t) \geq t^N \mu(\Omega)$ for all $N \geq k+3$.

\subsection{Step 2}
Fix $\varepsilon>0$. We prove that the $\mathrm{MCP}(0,k+3-\varepsilon)$ does not hold. Let $\lambda = (p_x,0) \in D$. By Lemma~\ref{l:jacobian}, and recalling that $J > 0$ on $D$, we have 
\begin{equation}
J(t p_x,0)= t^2 J(p_x,0) < t^{2-\varepsilon} J(p_x,0), \qquad \forall t \in[0,1].
\end{equation}
By continuity of $J$ and compactness of $[0,1]$, we find an open neighborhood $A \subset D$ of $\lambda$ such that $J(t\lambda) < t^{2-\varepsilon} J(\lambda)$, for all $t \in [0,1]$. In particular, for $\Omega= \exp_e(A)$, we obtain
\begin{equation}
\mu(\Omega_t) < t^{k+3-\varepsilon} \mu(\Omega), \qquad t \in [0,1].
\end{equation}

\subsection{Step 3}

To  prove that $\mathrm{MCP}(K,N)$ does not hold for $K>0$ and any $N >1$, we observe that spaces verifying this condition are bounded, while $G \simeq \R^n$ clearly is not.
Finally, assume that $(G,d,\mu)$ satisfies $\mathrm{MCP}(K,N)$ for some $K < 0$ and $N < k+3$. Then the scaled space $(G,\varepsilon^{-1} d, \varepsilon^{-Q} \mu)$ (where $\varepsilon>0$ and $Q = k+2$ is the Haudorff dimension of $(G,d)$) verifies $\mathrm{MCP}(\varepsilon^2 K,N)$ for \cite[Lemma 2.4]{Ohta-MCP}. But the two spaces $(G,d,\mu)$ and $(G,\varepsilon^{-1} d, \varepsilon^{-Q} \mu)$ are isometric through the dilation $\delta_\varepsilon(x,z):=(\varepsilon x, \varepsilon^2 z)$. In particular $(G,d,\mu)$ satisfies the $\mathrm{MCP}(\varepsilon K,N)$ for all $\varepsilon>0$, that is
\begin{equation}
\mu(\Omega_t) \geq \int_{\Omega} t \left[\frac{s_{\varepsilon K}(t d(x,z) / \sqrt{N-1})}{s_{\varepsilon K}(d(x,z) / \sqrt{N-1})}\right]^{N-1} d \mu(z), \qquad \forall t \in[0,1].
\end{equation}
Taking the limit for $\varepsilon \to 0^+$, we obtain that $(G,d,\mu)$ satisfies the $\mathrm{MCP}(0,N)$ with $N< k+3$, but this is false by the previous step  (these are the same arguments of \cite{Juillet}). $\hfill \qed$

\section{Proof of Theorem \ref{t:basicprop}}\label{s:proofbasicprop}

Assume that $B$ is bounded. In particular, $\mu(B) < + \infty$. For any $k>0$ let $\mathcal{H}^k$ denote the $k$-dimensional Hausdorff measure on $(X,d)$. Let $k < \dim_H(B)$, then $\mathcal{H}^k(B) = +\infty$. By \cite[Theorem 2.4.3]{AT-Metric}, there exists an $x \in B$ such that
\begin{equation}
\limsup_{t \to 0^+} \frac{\mu(\ball(x,t))}{t^k} < + \infty.
\end{equation}
Let $\Omega$ be a bounded measurable set with $0< \mu(\Omega) < +\infty$, and let $\Omega_t$ be its homothety with center $x$. We have $\Omega \subset \ball(x,R)$ for some $R>0$, and $\Omega_t \subseteq \ball(x,tR)$. In particular
\begin{equation}
\limsup_{t \to 0^+} \frac{1}{t^{k-\varepsilon}} \frac{\mu(\Omega_t)}{\mu(\Omega)} \leq \limsup_{t \to 0^+} \frac{1}{t^{k-\varepsilon}} \frac{\mu(\ball(x,tR))}{\mu(\Omega)} = 0, \qquad \forall \varepsilon>0.
\end{equation}
Since this holds for any bounded $\Omega$, we have $C_{k-\varepsilon}(x) = 0$ for all $k < \dim_H(B)$ and $\varepsilon>0$. By definition of geodesic dimension $\mathcal{N}(x) = \sup\{s>0 \mid C_s(x) = 0\} \geq k$. Thus, for any $k < \dim_H(B)$ we have found $x \in B$ such that $\mathcal{N}(x) \geq k$, which implies the statement. 

If $B$ is not bounded, consider the increasing sequence of bounded sets $B_j:= B\cap \ball(x,j)$, with $j \in \N$, and observe that $\dim_H(B_j)$ is a non-decreasing sequence for $j \to \infty$. $\hfill \qed$

\section{Proof of Theorem \ref{t:bound}}\label{s:proofbound}

By contradiction, assume that $N < \sup\{\mathcal{N}(x) \mid x \in M\}$. In particular there exists $x \in X$ such that $\mathcal{N}(x) > N$. Let $\Omega \subset X$ be a bounded, measurable set such that $0 < \mu(\Omega) <+\infty$, and with $\Omega \subset \ball(x,\pi\sqrt{N-1/K})$ if $K > 0$. By the $\mathrm{MCP}(K,N)$ we have
\begin{equation}
\frac{\mu(\Omega_t)}{\mu(\Omega)} \geq \frac{1}{\mu(\Omega)}\int_\Omega t \left[\frac{s_K(t d(x,z)/\sqrt{N-1})}{s_K(d(x,z)/\sqrt{N-1})}\right]^{N-1} d\mu(z), \qquad \forall t \in [0,1].
\end{equation}
We have $\Omega \subset \ball(x,R\sqrt{N-1})$ for some sufficiently large $R$ (with $R< \pi/\sqrt{K}$ if $K>0$). Consider the functions $s_K(t\delta)/ s_K(\delta)$, for $\delta \in (0,R)$. By explicit inspection using \eqref{eq:sturm} we find a constant $A_{K,R}>0$ (independent on $\delta$) such that
\begin{equation}
s_K(t\delta)/s_K(\delta) \geq A_{K,R} t, \qquad \forall t \in [0,1], \quad \forall \delta \in (0,R).
\end{equation}
Thus we have
\begin{equation}
\frac{\mu(\Omega_t)}{\mu(\Omega)} \geq \frac{1}{\mu(\Omega)}  \int_\Omega A_{K,R}^{N-1} t^N   d\mu(z) =A_{K,R}^{N-1} t^N , \qquad \forall t \in [0,1].
\end{equation}
Let $\mathcal{N}(x) - N = 2\varepsilon > 0$. We have
\begin{equation}
\limsup_{t \to 0^+} \frac{1}{t^{\mathcal{N}(x)-\varepsilon}} \frac{\mu(\Omega_t)}{\mu(\Omega)}\geq \lim_{t \to 0^+} \frac{A_{K,R}^{N-1}}{t^{\mathcal{N}(x)- \varepsilon -N}}  =   \lim_{t \to 0^+} \frac{A_{K,R}^{N-1}}{t^{\varepsilon}}  = +\infty.
\end{equation}
In particular $C_{\mathcal{N}(x)-\varepsilon}(x) = +\infty$, where $C_s(x)$ is defined in \eqref{eq:criticalratio}. This is a contradiction since $\mathcal{N}(x) = \inf\{s >0 \mid C_s(x) = +\infty \}$. $\hfill \qed$

\section{Formula for the geodesic dimension}\label{s:geod}

We recall some of the results of \cite{ABR-Curvature} and we prove that the definition of geodesic dimension given in this paper coincides with the one of \cite{ABR-Curvature} for sub-Riemannian structures.

\subsection{Flag of the distribution and Hausdorff dimension}

Let $(M,\distr,g)$ be a  fixed (sub-)Riemannian structure. The \emph{flag of the distribution} at $x \in M$ is the filtration of vector subspaces $\distr_x^1 \subseteq \distr_x^2 \subseteq \ldots \subseteq T_x M$ defined as
\begin{equation}
\distr_x^1:= \distr_x, \qquad \distr_x^{i+1} := \distr_x^i + [\distr,\distr^i]_x,
\end{equation}
where $[\distr,\distr^i]_x$ is the vector space generated by the iterated Lie brackets, up to length $i+1$, of local sections of $\distr$, evaluated at $x$. We denote with $s_x$ the \emph{step of the distribution} at $x$, that is the smallest (finite) integer such that $\distr_x^{s_x} = T_x M$.

We say that $\distr$ is \emph{equiregular} if $\dim \distr_x^i$ are constant for all $i\geq 0$. In this case the step is constant and equal to $s$. The \emph{growth vector of the distribution} is
\begin{equation}
(d_1,\ldots,d_s), \qquad d_i:=\dim\distr^i.
\end{equation}

\begin{theorem}[Mitchell \cite{Mitchell}] Let $(M,\distr,g)$ an equiregular (sub-)Riemannian structure. Then its Hausdorff dimension is given by the following formula:
\begin{equation}\label{eq:mitchell}
\dim_H(M) = \sum_{i=1}^s i (d_i-d_{i-1}), \qquad d_0 := 0.
\end{equation}
\end{theorem}

\subsection{Flag of the geodesic and geodesic dimension}

Let $\gamma_\lambda :[0,\varepsilon) \to M$ be a normal geodesic, with initial covector $\lambda$, and $x = \gamma(0)$. Let $T \in \Gamma(\distr)$ any horizontal extension of $\gamma$, that is $T(\gamma(t)) = \dot{\gamma}(t)$ for all $t \in [0,\varepsilon)$. The \emph{flag of the geodesic} is the filtration of vector subspaces $\mathcal{F}^1_\lambda \subseteq \mathcal{F}^2_\lambda \subseteq \ldots \subseteq T_x M$ defined by
\begin{equation}
\mathcal{F}^i_\lambda := \spn\{\mathcal{L}_T^j(X)|_{x} \mid X \in \Gamma(\distr), \quad j \leq i-1\} \subseteq T_{x} M, \qquad i \geq 1,
\end{equation}
where $\mathcal{L}$ denotes the Lie derivative.
By \cite[Section 3.4]{ABR-Curvature}, this definition does not depend on the choice of the extension $T$, but only on the germ of $\gamma(t)$ at $t=0$. In particular, it depends only on the initial covector $\lambda \in T_x^*M$. We define the \emph{geodesic growth vector} as
\begin{equation}
\mathcal{G}_\lambda := (k_1,\ldots,k_i,\ldots ), \qquad k_i := \dim \mathcal{F}^i_\lambda.
\end{equation}
We say that $\gamma_\lambda$ is \emph{ample} (at $t=0$) if there is a smallest integer $m \geq 1$ such that $\mathcal{F}^m_\lambda = T_x M$. In this case the growth vector is constant after its $m$-th entry, and $m$ is called the \emph{geodesic step}. Different initial covectors may give different growth vectors (possibly associated with non-ample geodesics when $\gamma$ is abnormal). The \emph{maximal geodesic growth vector} at $x$ is
\begin{equation}
\mathcal{G}_x^{max} := (k_1^{max},\ldots, k_i^{max},\ldots ), \qquad k_i^{max} := \max \{\dim \mathcal{F}^i_\lambda \mid \lambda \in T_x^*M\}.
\end{equation}

\begin{theorem}\label{t:curv1}
The set $\mathcal{A}_x \subset T_x^*M$ of initial covectors such that the corresponding geodesic is ample, and its growth vector is maximal is an open, non-empty Zariski subset.
\end{theorem}
In particular, the generic normal geodesic starting at $x$ has maximal growth vector and the minimal step $m(x)$. For a fixed $x \in M$, consider the following number:
\begin{equation}\label{eq:gdformula}
\mathcal{N}(x) = \sum_{i=1}^{m(x)} (2i-1)(k_i^{max}-k_{i-1}^{max}), \qquad k_0^{max}:= 0.
\end{equation}
\begin{theorem}\label{t:curv2}
Let $(M,\distr,g)$ be a sub-Riemannian manifold, equipped with a smooth measure $\mu$. Assume that, as a metric measure space $(M,d,\mu)$, has negligible cut loci. Let $x \in M$, and let $\Omega$ be any measurable, bounded subset with $0 < \mu(\Omega) < +\infty$. Then there exists a constant $C(\Omega) > 0$ such that
\begin{equation}
\mu(\Omega_t) \sim C(\Omega)t^{\mathcal{N}(x)}, \qquad t \to 0^+.
\end{equation}
\end{theorem}
Equation \eqref{eq:gdformula} is the \emph{definition} of geodesic dimension given in \cite{ABR-Curvature}. As a consequence of Theorem~\ref{t:curv2}, it coincides with the one given in this paper, when specified to sub-Rieman\-nian structures. In this case, to compute $\mathcal{N}(x)$, it is sufficient to compute the growth vector for the generic geodesic, and use \eqref{eq:gdformula}. Theorems~\ref{t:curv1}, \ref{t:curv2} are proved in \cite{ABR-Curvature}, and are based on a deep relation between the geodesic growth vector and the asymptotics of the exponential map on a general sub-Riemannian manifold.

\subsection{Proof of Theorem \ref{t:basicprop2}}
If $(M,\distr,g)$ is Riemannian, for any point $x \in M$ we have $\mathcal{G}_x = (\dim(M))$ for any non-trivial initial covector and $\mathcal{N}(x) = \dim(M )= \dim_H (M)$. 

If $(M,\distr,g)$ is sub-Riemannian (with $k = \rank \distr < n$), and equiregular of step $s$, let $q_i := d_i-d_{i-1}$, for $i=1,\ldots,s$ and $p_i:= k_i^{max}-k_{i-1}^{max}$, for all $i=1,\ldots,m(x)$. Observe that $m(x) \geq s \geq 2$. By Mitchell's formula \eqref{eq:mitchell}, and \eqref{eq:gdformula}, we have
\begin{align}
\dim_H(M) & = \underbracket{1+\dots+1}_{q_1} + \underbracket{2 + \dots + 2}_{q_2}+ \dots + \underbracket{s + \dots +s}_{q_s}, \label{eq:sum1} \\
\mathcal{N}(x)& = \underbracket{1+\dots+1}_{p_1}+\underbracket{3+\dots+3}_{p_2}+\dots+\underbracket{2m(x)-1+\dots+2m(x)-1}_{p_{m(x)}}. \label{eq:sum2}
\end{align}
Both sums have a total of $n$ terms, in fact
\begin{equation}
\sum_{i=1}^s q_i = \sum_{i=1}^s d_i -d_{i-1} = d_s = n, \qquad \sum_{i=1}^{m(x)}p_i = \sum_{i=1}^{m(x)} k_i^{max} -k_{i-1}^{max} = k_{m(x)}^{max} = n.
\end{equation}
Moreover $q_1 = p_1 = k<n$. The terms of \eqref{eq:sum2}, after the $k$-th, are strictly greater then the ones of \eqref{eq:sum1}. Thus, $\mathcal{N}(x) > \dim_H(M) > \dim(M)$.  $\hfill \qed$

\subsection{The Engel group}\label{s:Engel}

We discuss more in detail the Engel group introduced in Example~\ref{ex:engel}. This is the Carnot group, in dimension $n=4$, generated by the following global orthonormal left-invariant vector fields in coordinates $(x_1,x_2,x_3,x_4) \in \R^4$
\begin{equation}
X_1 = \partial_1, \qquad X_2 = \partial_2 + x_1 \partial_3 + x_1 x_2 \partial_4.
\end{equation}
It is a rank $2$ Carnot group of step $3$, with $\mathfrak{g}_1 = \spn\{X_1,X_2\}$ and 
\begin{equation}
\mathfrak{g}_2 = [X_1,X_2]=\partial_3 + x_2 \partial_4, \qquad \mathfrak{g}_3 = [X_2,[X_1,X_2]]= \partial_4,
\end{equation}
where we omit the linear span. In particular, by left-invariance, $\distr^1 = \mathfrak{g}_1$, $\distr^2 = \mathfrak{g}_1 \oplus \mathfrak{g}_2$ and $\distr^3 = \mathfrak{g}_1 \oplus \mathfrak{g}_2 \oplus \mathfrak{g}_3$. The growth vector of the distribution is $(2,3,4)$. By Mitchell's formula~\eqref{eq:mitchell} for the Hausdorff dimension we have $Q = 2 + 2 + 3 = 7$.

Let us compute the geodesic growth vector. As we will see, it is sufficient to choose the curve $\gamma(t) = e^{tX_2}(e)$ (this is a normal geodesic, by Lemma~\ref{l:char}). Using the definition, we obtain $\mathcal{F}^1 =  \spn\{X_1,X_2\}$ and, omitting the linear span,
\begin{equation}
\mathcal{F}^2 = [X_2,X_1] = \partial_3 + x_2 \partial_4, \qquad \mathcal{F}^3  = [X_2,[X_2,X_1]]=  \partial_4.
\end{equation}
This gives the maximal possible geodesic growth vector, hence
\begin{equation}
\mathcal{G}^{max}_x = (2,3,4), \qquad \forall x \in G.
\end{equation}
In particular, using~\eqref{eq:gdformula} we obtain $\mathcal{N} = 2 + 3 + 5 = 10$.
\begin{rmk}\label{r:Engelcut}
We stress that the Engel group has negligible cut loci, hence it falls into the class of metric measure spaces considered in this article. This follows from the fact that in step $3$ Carnot groups all minimizing geodesics are normal and, in particular, the set of points reached by strictly abnormal geodesics, starting from the origin, has zero measure (see \cite[Theorem 5.4]{TYStep3} and \cite[Theorem 1.5]{DMOPV-Sard} for an independent proof).
\end{rmk}

\section{Proof of Theorem \ref{p:fat}}\label{s:fat}

Let $(G,\distr,g)$ be a Carnot group of step $s$ and dimension $n$. We identify $\mathfrak{g} = T_e G$ (and its subspaces) with the vector space of left-invariant vector fields. In particular $\distr = \mathfrak{g}_1$.

\subsection{Step 1: Estimates in the fat case}
We remind that a sub-Riemannian structure $(M,\distr,g)$ is \emph{fat} (or \emph{strong bracket-generating}) if for any $x \in M$ and $X \in \distr$, $X(x) \neq0$, then $\distr_x + [X,\distr]_x=T_xM$. It is well known that fat structures does not admit non-trivial abnormal minimizers \cite{nostrolibro,montgomerybook,rifford2014sub}. If $G$ is a fat Carnot group of rank $k$, then the geodesic growth vector of any non-trivial geodesic is
\begin{equation}
\mathcal{G}_x^{max} = (k,n), \qquad \forall x \in G.
\end{equation}
By~\eqref{eq:gdformula} we have $\mathcal{N} = k + 3 (n-k) = 3n-2k$. On the other hand, the Hausdorff dimension is $Q = k + 2(n-k) = 2n-k$. Moreover, from \eqref{eq:boundRiff}, we have $N_R = Q + n -k = 3n-2k$. This proves that on a fat Carnot group $\mathcal{N} = N_R$.

To prove the inequality $\mathcal{N} > N_R$ when $G$ has step $s>2$, let $\mathcal{G}^{max} = (k_1,k_2,\ldots,k_m)$ be the maximal geodesic growth vector, with geodesic step $m \geq s > 2$. Let $q_i:=d_i-d_{i-1}$, for $i=1,\ldots,s$ and $p_i:=k_i - k_{i-1}$, for $i=1,\ldots,m$. Since $d_0 = k_0 = 0$ by convention and $d_s = k_m = n$, we have $n = \sum_{i=1}^s q_i = \sum_{i=1}^m p_i$. Thus, from \eqref{eq:gdformula}, we obtain
\begin{equation}
\mathcal{N} = \sum_{i=1}^m (2i-1)p_i = \sum_{i=2}^m (2i-2)p_i + n.
\end{equation}
On the other hand, for $N_R = Q+n -k$ and using Mitchell's formula \eqref{eq:mitchell}, we obtain
\begin{equation}
N_R = \sum_{i=1}^s i q_i + n - k = \sum_{i=2}^s i q_i +n .
\end{equation}
Arranging the terms as we did in the proof of Theorem~\ref{t:basicprop2}, we write
\begin{align}
\mathcal{N} - n   &= \underbracket{2 + \dots + 2}_{p_2} + \underbracket{4 + \dots + 4}_{p_3} + \dots +\underbracket{2m-2 + \dots + 2m-2}_{p_m}, \label{eq:sum3}\\ 
N_R - n & = \underbracket{2 + \dots + 2}_{q_2} + \underbracket{3 + \dots + 3}_{q_3} + \dots +\underbracket{s + \dots + s}_{p_s} . \label{eq:sum4}
\end{align}
Both sums have $n-k=\sum_{i=2}^m p_i = \sum_{i=2}^s q_i$ entries. Since $m \geq s >2$, the entries of \eqref{eq:sum3}, after the $p_2$-th one, are strictly greater than the corresponding ones of \eqref{eq:sum4}, and $\mathcal{N} > N_R$.

\subsection{Step 2: Ideal = Fat}

To conclude the proof of Theorem~\ref{p:fat}, we prove that any ideal Carnot group is fat. Denote with $\mathrm{ad}_X : \mathfrak{g} \to \mathfrak{g}$ the linear map:
\begin{equation}
\mathrm{ad}_X(Y) := [X,Y] =\left.\frac{d}{dt}\right|_{t=0} e^{-t X}_* Y(\gamma_X(t)),
\end{equation}
where $\gamma_X(t)=e^{tX}(x)$ is the integral curve of $X \in \mathfrak{g}_1$ starting from $x \in G$.

\begin{lemma}\label{l:char}
Let $\gamma_X(t)=e^{tX}(x)$ be the integral curve of the left-invariant vector field $X \in \mathfrak{g}_1$, starting from $x$. Then $\gamma_X$ it is a normal geodesic. It is also an abnormal geodesic if and only if there exists a non-zero $\lambda \in T_e^*G$ such that
\begin{equation}\label{eq:charabn}
\langle \lambda, \mathrm{ad}_X^i(\mathfrak{g}_1)\rangle = 0, \qquad \forall i=0,\ldots,s-1.
\end{equation}
\end{lemma}
\begin{proof}
Let $X_1,\ldots,X_k$ be a basis of left-invariant vector fields. Clearly, $X = \sum_{i=1}^k u_i X_i$ for a constant control $u \in L^\infty([0,1],\R^k)$. By left-invariance we can set $x = e$. A well known formula for the differential of the end-point map \cite{nostrolibro,rifford2014sub}  $D_u E_e : T_u \mathcal{U} \simeq \mathcal{U} \to T_{\gamma(1)} G$, gives
\begin{equation}
D_u E_e(v) = \int_0^1 e^{(1-t) X}_* \sum_{i=1}^k v_i(t) X_i(\gamma(t))dt, \qquad \forall v \in L^\infty([0,1],\R^k).
\end{equation}
We first prove that $\gamma_X(t)$ is a normal geodesic. Consider the covector $\eta \in T_{e} G$ such that $\langle\eta ,X_i\rangle = u_i$ and $\langle \eta, W\rangle = 0$ for all $W \in \mathfrak{g}_2 \oplus \ldots \oplus \mathfrak{g}_s$. Then
\begin{equation}
\langle (e^{-X})^* \eta, D_u E_e (v)\rangle =  \int_{0}^1 \sum_{i=1}^k v_i(t) \langle \eta, e^{-t X}_* X_i(\gamma(t))\rangle  dt = (u,v)_{L^2([0,1],\R^k)} = D_u J (v).
\end{equation}
Thus $\gamma_X(t)$ with control $u$ satisfies the normal Lagrange multiplier rule with covector $\eta_1=(e^{-X})^*\eta \in T_{\gamma(1)}^*G$, and is a normal geodesic. By definition $\gamma_X(t)$ is also abnormal if and only if there exists a $\lambda_1 \in T_{\gamma(1)}^* G$ such that $\lambda_1 \circ D_u E_e = 0$. That is, if and only if there exists $\lambda = (e^{X})^* \lambda_1 \in T_e^*G$ such that
\begin{equation}
0  = \langle \lambda, e^{-X}_* D_u E_e(v) \rangle  = \int_0^1 \sum_{i=1}^k v_i(t)  \langle \lambda, e^{-tX}_* X_i(\gamma(t)) \rangle dt, \qquad \forall v \in L^\infty([0,1],\R^k).
\end{equation}
This is true if and only if $\langle\lambda, e^{-tX}_* Y \rangle = 0$ for any $Y \in \mathfrak{g}_1$ and $t \in [0,1]$. Since all the relevant data are analytic, $t \mapsto \langle\lambda, e^{-tX}_* Y \rangle $ is an analytic function of $t$. Hence it vanishes if and only if all its derivatives at $t=0$ are zero, and this condition coincides with \eqref{eq:charabn}.
\end{proof}

\begin{lemma}\label{l:step2}
Let $G$ be a Carnot group of step $s \leq 2$. $G$ is ideal if and only if it is fat.
\end{lemma}
\begin{proof}
The implication fat $\Rightarrow$ ideal is trivial. Then, assume that $\mathfrak{g}_1$ is not fat, i.e.\ there exists $X \neq 0 \in \mathfrak{g}_1$ such that $\mathfrak{g}_1 \oplus [X,\mathfrak{g}_1] \subsetneqq T_e G$. Hence there exists $\lambda \neq 0 \in T_e^*G$ such that
\begin{equation}
0 = \langle \lambda, \mathfrak{g}_1 \rangle = \langle \lambda, \mathrm{ad}_X (\mathfrak{g}_1)\rangle.
\end{equation}
By Lemma~\ref{l:char}, $\gamma_X(t) = e^{tX}(e)$ is a normal and abnormal geodesic. In particular, a sufficiently short segment of it is a minimizing curve.
\end{proof}
We learned the following fact by E. Le Donne. For the reader's convenience we provide a simple proof here, which is similar to the one in \cite{Enrico-Regularity}.
\begin{lemma}\label{l:step3noideal}
Let $G$ be a Carnot group of step $s \geq 3$. Then there exists a non-zero $X \in \mathfrak{g}_1$ such that the integral curve $\gamma_X(t)=e^{t X}(e)$ is a normal and abnormal geodesic.
\end{lemma}
\begin{proof}
If there exists a $X \neq 0 \in \mathfrak{g}_1$ such that $\mathrm{ad}_X(\mathfrak{g}_1) \subsetneqq \mathfrak{g}_2$, using the same argument of Lemma~\ref{l:step2}, we show that $\gamma_X(t)$ is abnormal and normal. Let $q_i := \dim \mathfrak{g}_i$, for $i=1,\ldots,s$. Then assume that, for any $X\neq 0 \in \mathfrak{g}_1$, we have $\mathrm{ad}_X(\mathfrak{g}_1) = \mathfrak{g}_2$. This implies $q_2 \leq q_1-1$. Consider a basis $Y_1,\ldots,Y_{q_2}$ of $\mathfrak{g}_2$ and a basis $Z_1,\ldots,Z_{q_3}$ of $\mathfrak{g}_3$. Let $\lambda \in T_e^*G$ such that $\langle \lambda,Z_i\rangle = 0$ if $i >1$ and $\langle \lambda,Z_1\rangle = 1$, while $\langle\lambda,\mathfrak{g}_i\rangle = 0$ for all $i \neq 3$. Consider the linear maps $A_i:=\lambda \circ \mathrm{ad}_{Y_i} : \mathfrak{g}_1 \to \R$, for $i=1,\ldots,q_2$. We have $\dim \ker A_i \geq q_1 -1$. Moreover
\begin{align}
\dim (\ker A_1 \cap \ker A_2 ) & = \dim \ker A_1 + \dim \ker A_2 - \dim (\ker A_1 + \ker A_2) \\
& \geq 2(q_1 - 1) - q_1 = q_1 -2.
\end{align}
After a finite number of similar steps we arrive to
\begin{equation}
\dim(\ker A_1 \cap \ldots \cap \ker A_{q_2}) \geq q_1 -q_2 \geq q_1 -( q_1 - 1) = 1.
\end{equation}
Thus let $X \neq 0 \in \ker A_1 \cap \ldots \cap \ker A_{q_2}$. We show that $\gamma_X(t) = e^{t X}(e)$, which is a normal geodesic, verifies the abnormal characterization of Lemma~\ref{l:char} with covector $\lambda$. Since $\mathrm{ad}_X^i(\mathfrak{g}_1) \subseteq \mathfrak{g}_{i+1}$, we have $\langle\lambda,\mathrm{ad}_X^i (\mathfrak{g}_1)\rangle = 0$ for all $i \neq 2$ by construction of $\lambda$. Finally,
\begin{equation}
\langle \lambda, \mathrm{ad}_X^2(\mathfrak{g}_1) \rangle = \langle \lambda,\mathrm{ad}_X(\mathfrak{g}_2)\rangle  = 0,
\end{equation}
where in the last passage we used the fact that $\mathrm{ad}_X(\mathfrak{g}_1) = \mathfrak{g}_2$, that the latter is generated by the $Y_i$, and the definition of $X$. Then $\gamma_X(t)$ is abnormal by Lemma~\ref{l:char}.
\end{proof}
To conclude the proof, recall that Carnot groups are complete. Since fat structures do not admit non-trivial abnormal curves \cite[Section 5.6]{montgomerybook}, fat $\Rightarrow$ ideal. Moreover, ideal $\Rightarrow$ step $s \leq 2$ (Lemma~\ref{l:step3noideal}). On the other hand, ideal and step $s \leq 2$ $\Rightarrow$ fat (Lemma~\ref{l:step2}). $\hfill \qed$

\bibliographystyle{abbrv}
\bibliography{MCPc1}

\begin{thebibliography}{10}

\bibitem{ABB-Hausdorff}
A.~Agrachev, D.~Barilari, and U.~Boscain.
\newblock On the {H}ausdorff volume in sub-{R}iemannian geometry.
\newblock {\em Calc. Var. Partial Differential Equations}, 43(3-4):355--388,
  2012.

\bibitem{nostrolibro}
A.~Agrachev, D.~Barilari, and U.~Boscain.
\newblock Introduction to {R}iemannian and sub-{R}iemannian geometry ({L}ecture
  {N}otes).
\newblock v11/07/15.

\bibitem{ABR-Curvature}
A.~{Agrachev}, D.~{Barilari}, and L.~{Rizzi}.
\newblock {Curvature: a variational approach}.
\newblock {\em Memoirs of the AMS (in press)}.

\bibitem{Agrabook}
A.~Agrachev and Y.~L. Sachkov.
\newblock {\em Control theory from the geometric viewpoint}, volume~87 of {\em
  Encyclopaedia of Mathematical Sciences}.
\newblock Springer-Verlag, Berlin, 2004.
\newblock Control Theory and Optimization, II.

\bibitem{AT-Metric}
L.~Ambrosio and P.~Tilli.
\newblock {\em Topics on analysis in metric spaces}, volume~25 of {\em Oxford
  Lecture Series in Mathematics and its Applications}.
\newblock Oxford University Press, Oxford, 2004.

\bibitem{AS-Engel}
A.~A. Ardentov and Y.~L. Sachkov.
\newblock Conjugate points in nilpotent sub-{R}iemannian problem on the {E}ngel
  group.
\newblock {\em J. Math. Sci. (N. Y.)}, 195(3):369--390, 2013.
\newblock Translation of Sovrem. Mat. Prilozh. No. 82 (2012).

\bibitem{BR-Popp}
D.~Barilari and L.~Rizzi.
\newblock A formula for {P}opp's volume in sub-{R}iemannian geometry.
\newblock {\em Anal. Geom. Metr. Spaces}, 1:42--57, 2013.

\bibitem{GJ-Hausdorff}
R.~Ghezzi and F.~Jean.
\newblock Hausdorff volume in non equiregular sub-{R}iemannian manifolds.
\newblock {\em Nonlinear Anal.}, 126:345--377, 2015.

\bibitem{MR2307192}
M.~Gromov.
\newblock {\em Metric structures for {R}iemannian and non-{R}iemannian spaces}.
\newblock Modern Birkh\"auser Classics. Birkh\"auser Boston, Inc., Boston, MA,
  english edition, 2007.
\newblock Based on the 1981 French original, With appendices by M. Katz, P.
  Pansu and S. Semmes, Translated from the French by Sean Michael Bates.

\bibitem{Juillet}
N.~Juillet.
\newblock Geometric inequalities and generalized {R}icci bounds in the
  {H}eisenberg group.
\newblock {\em Int. Math. Res. Not. IMRN}, (13):2347--2373, 2009.

\bibitem{DMOPV-Sard}
E.~{Le Donne}, R.~Montgomery, A.~Ottazzi, P.~Pansu, and D.~Vittone.
\newblock Sard property for the endpoint map on some {C}arnot groups.
\newblock {\em Annales de l'Institut Henri Poincare (C) Non Linear Analysis (in
  press)}, 2015.

\bibitem{Enrico-Regularity}
E.~{Le Donne} and S.~{Nicolussi Golo}.
\newblock {Regularity properties of spheres in homogeneous groups}.
\newblock {\em ArXiv e-prints}, Sept. 2015.

\bibitem{LV-Ricci}
J.~Lott and C.~Villani.
\newblock Ricci curvature for metric-measure spaces via optimal transport.
\newblock {\em Ann. of Math. (2)}, 169(3):903--991, 2009.

\bibitem{Mitchell}
J.~Mitchell.
\newblock On {C}arnot-{C}arath\'eodory metrics.
\newblock {\em J. Differential Geom.}, 21(1):35--45, 1985.

\bibitem{montgomerybook}
R.~Montgomery.
\newblock {\em A tour of subriemannian geometries, their geodesics and
  applications}, volume~91 of {\em Mathematical Surveys and Monographs}.
\newblock American Mathematical Society, Providence, RI, 2002.

\bibitem{Ohta-MCP}
S.-i. Ohta.
\newblock On the measure contraction property of metric measure spaces.
\newblock {\em Comment. Math. Helv.}, 82(4):805--828, 2007.

\bibitem{Rifford}
L.~Rifford.
\newblock Ricci curvatures in {C}arnot groups.
\newblock {\em Math. Control Relat. Fields}, 3(4):467--487, 2013.

\bibitem{rifford2014sub}
L.~Rifford.
\newblock {\em Sub-{R}iemannian geometry and optimal transport}.
\newblock Springer Briefs in Mathematics. Springer, Cham, 2014.

\bibitem{RT-MorseSard}
L.~Rifford and E.~Tr{\'e}lat.
\newblock Morse-{S}ard type results in sub-{R}iemannian geometry.
\newblock {\em Math. Ann.}, 332(1):145--159, 2005.

\bibitem{S-I}
K.-T. Sturm.
\newblock On the geometry of metric measure spaces. {I}.
\newblock {\em Acta Math.}, 196(1):65--131, 2006.

\bibitem{S-II}
K.-T. Sturm.
\newblock On the geometry of metric measure spaces. {II}.
\newblock {\em Acta Math.}, 196(1):133--177, 2006.

\bibitem{TYStep3}
K.~Tan and X.~Yang.
\newblock Subriemannian geodesics of {C}arnot groups of step 3.
\newblock {\em ESAIM Control Optim. Calc. Var.}, 19(1):274--287, 2013.

\end{thebibliography}

\end{document}